\documentclass[10pt, reqno]{amsart}
\usepackage{graphicx, amssymb, amsmath, amsthm, color, cite, subfig, hyperref}
\numberwithin{equation}{section}

\let\Re=\undefined\DeclareMathOperator*{\Re}{Re}

\newcommand{\R}{\mathbb{R}}
\newcommand{\C}{\mathbb{C}}

\newcommand{\eps}{\varepsilon}

\renewcommand{\S}{\mathcal{S}}

\newcommand{\F}{\mathcal{F}}

\newtheorem{theorem}{Theorem}[section]

\newtheorem*{theorem*}{Theorem}

\newtheorem{lemma}[theorem]{Lemma}

\newtheorem{proposition}[theorem]{Proposition}

\theoremstyle{definition}
\newtheorem{definition}[theorem]{Definition}

\theoremstyle{remark}

\newcommand{\qtq}[1]{\quad\text{#1}\quad}

\begin{document}

\title[Determining nonlocal nonlinearities]{Determination of radial nonlocal nonlinearities \\ from the scattering map} 

\author[L. Campos]{Luccas Campos}
\address{Department of Mathematics, Universidade Federal de Minas Gerais}
\email{luccascampos@gmail.com}

\author[R. Killip]{Rowan Killip}
\address{CNRS CEREMADE}
\email{killip@ceremade.dauphine.fr}

\author[J. Murphy]{Jason Murphy}
\address{Department of Mathematics, University of Oregon}
\email{jamu@uoregon.edu}

\author[M. Vi\c{s}an]{Monica Vi\c{s}an}
\address{Department of Mathematics, UCLA}
\email{visan@math.ucla.edu} 

\maketitle

\begin{abstract}
We show that the small-data scattering map uniquely determines the nonlinearity for a class of nonlinear Schr\"odinger equations with radial, Hartree-type nonlinearities.  Our assumptions on the convolution kernel require only a mild decay condition at infinity and permit a locally integrable singularity at the origin.
\end{abstract}

\section{Introduction}

We consider nonlinear Schr\"odinger equations with nonlocal, Hartree-type nonlinearities in dimensions $d\geq 3$:
\begin{equation}\label{nls}
(i\partial_t + \Delta) u = (K\ast |u|^2)u,
\end{equation}
where $u:\R_t\times\R_x^d\to\C$, $\ast$ denotes convolution in the spatial variable, and $K:\R^d\backslash\{0\}\to\R$.  Our assumptions on the convolution kernel $K$ are detailed in the following definition. 

\begin{definition}[Admissible kernels]\label{D:admissible} Let $d\geq 3$.  We say a function $K:\R^d\backslash\{0\}\to\R$ is \emph{admissible with parameter $c\geq 2$} if $K$ is measurable and there exists $C>0$ such that
\begin{equation}\label{Kbds}
|K(x)| \leq C\bigl(|x|^{-2}+|x|^{-c}\bigr)\qtq{for a.e.}x\in\R^d\backslash\{0\}.
\end{equation}
\end{definition}

%\begin{remark}\label{R:nesting} If $K$ is admissible with parameter $c\geq 2$, then $K$ is also admissible with parameter $c'$ for any $c'>c$. 
%\end{remark}
%

As the scaling-critical regularity for the standard Hartree equation
\[
(i\partial_t + \Delta)u = (|x|^{-a} \ast |u|^2) u
\]
is given by $s_a=\tfrac{a}{2}-1$, the bounds in Definition~\ref{D:admissible} can be interpreted as restricting the scaling-critical regularity of our model to the range $[0,\tfrac{c}{2}-1]$ (where $c\geq 2$ is the parameter associated with $K$).  To obtain a small-data scattering theory, we must limit the singularity of $K$ at the origin so that $K$ is locally integrable, which corresponds to choosing $c<d$.

In Theorem~\ref{T:scatter} below, we will prove that for any admissible $K$ with parameter $c\in[2,d)$, we may define a scattering map $\S$ on a small ball around zero in $H^{\frac{c}{2}-1}(\R^d)$, which maps a scattered state $u_-$ at $t=-\infty$ to a scattered state $u_+$ at $t=+\infty$ through equation \eqref{nls} (see Definition~\ref{D:scatter}). Our main result states that if the kernel $K$ is radial, then the small-data scattering map uniquely determines $K$.

\begin{theorem}\label{T} Let $K_1,K_2$ be radial and admissible in the sense of Definition~\ref{D:admissible} with parameters $2\leq c_1\leq c_2<d$. Let $\S_1,\S_2$ be the corresponding scattering maps, defined on a sufficiently small ball $\mathcal{B}\subset H^{\frac{c_2}{2}-1}(\R^d)$.  

If $\S_1|_{\mathcal{B}} \equiv \S_2|_{\mathcal{B}}$, then $K_1=K_2$ almost everywhere. 
\end{theorem}

The problem of determining an unknown nonlinearity from scattering data for dispersive PDE has been a topic of interest for several decades, and has received renewed interest in the last several years.  To provide some context for Theorem~\ref{T} and to properly highlight the contribution of this work, we begin by reviewing some related past results, focusing on `time-dependent' scattering problems in dispersive PDE. 

Early works on the recovery of the nonlinearity from the scattering map relied on fairly strong assumptions on the nonlinearity, such as analyticity or other structural assumptions.  Typical examples include the recovery of the parameters $(\alpha,p)$ in a local nonlinearity of the form $\alpha |u|^p u$ (where $\alpha$ may either be constant or a function of the spatial variable), or the recovery of the parameters $(\alpha,\gamma)$ in a nonlocal nonlinearity of the form $\alpha(|x|^{-\gamma}\ast |u|^2)u$. See, for example, \cite{CarlesGallagher, EW, MorStr, PauStr, Sasaki2, Sasaki3, Strauss, Watanabe0, Watanabe, Weder0, Weder1, Weder6, Weder3, Weder4, Weder5, Murphy, SasakiWatanabe, Sasaki4}. 

More recent work has addressed the case of significantly more general nonlinearities (in the local setting) and has also extended the theory in other directions, e.g. by proving stability estimates for nonlinearity recovery or recovering long-range nonlinearities using the modified scattering map (see e.g. \cite{ChenMurphy1, ChenMurphy2, ChenMurphy3, KMV, KMV2, SBUW, HKV}).

The contribution of the present work is to prove that the scattering map determines the nonlinearity for a significantly more general class of nonlocal nonlinearities than has been previously addressed.  

The results that are most closely-related to the present work are those of Sasaki--Watanabe \cite{SasakiWatanabe} and Sasaki \cite{Sasaki3}.   In \cite{SasakiWatanabe}, the authors considered systems for an unknown $u:\R_t\times\R_x^d\to\C^N$ of the form
\begin{equation}\label{SW}
(i\partial_t + \Delta) u = \int_{\R^d} K(x-y) U(t,x,y) \bar u(t,y)\,dy,
\end{equation}
where $U$ is the $N\times N$ matrix with entries
\[
U_{jk}(t,x,y) = \begin{cases} u_j(t,x)u_k(t,y), & \text{if } j\neq k, \\ 0, &  \text{if } j=k.\end{cases}
\]
Assuming (i) $K$ and its first derivative are bounded and continuous, (ii) $|K(x)|\leq C|x|^{-\sigma}$ for $2\leq\sigma\leq\min\{4,d-\}$, and (iii) $\widehat K$ is continuous, Sasaki and Watanabe proved that the small-data scattering map for \eqref{SW} uniquely determines the potential $K$.  

By writing the (scalar) Klein--Gordon-type equation 
\begin{equation}\label{KG}
\partial_t^2 w = \Delta w - w + (K\ast |w|^2)w
\end{equation}
as a first-order system, the authors of \cite{SasakiWatanabe} additionally proved that the small-data scattering map determines the potential $K$ in the setting of \eqref{KG} (with further structural hypotheses on $K$ in the setting of real-valued solutions).  In the setting of scalar NLS-type equations of the form \eqref{nls}, however, the methods of \cite{SasakiWatanabe} only determine the value of $\widehat K(0)$ (see \cite[Section~5]{SasakiWatanabe}).  

Building on the techniques of \cite{SasakiWatanabe}, Sasaki showed that the small-data scattering map for \eqref{nls} uniquely determines the potential $K$ in the setting of rapidly decaying potentials in three space dimensions \cite{Sasaki3}.  More precisely, Sasaki considered potentials $K$ satisfying $e^{A|x|}K \in L^1(\R^3)$ for some $A>0$, for which one has analyticity of $\widehat K$ on $\R^3$. For such potentials, Sasaki proved that the scattering map uniquely determines $\partial^\alpha \widehat K(0)$ for any multiindex $\alpha$ (with a quantitative estimate, in fact), and hence the scattering map determines $K$ itself.

In this paper, we address the case of radial convolution kernels satisfying the significantly weaker conditions appearing in Definition~\ref{D:admissible}, which require only a mild decay condition at infinity and permit a locally integrable singularity at the origin.

Like many of the works referenced above, to prove our main result we rely fundamentally on the \emph{Born approximation} to the scattering map.  This approximation entails replacing the full solution with the first Picard iterate in the implicit formula for the scattering map.  Specifically, the Duhamel formula for \eqref{nls} shows that the scattering map $\mathcal{S}$ associated to the kernel $K$ is given by 
\[
\S(\varphi) = \varphi - i\int_{\R} e^{-it\Delta}\bigl\{[K\ast |u(t)|^2]u(t)\bigr\}\,dt,
\] 
where $u$ is the solution to \eqref{nls} that scatters to $\varphi$ as $t\to-\infty$.  In the small-data regime, this operator is well-approximated by the functional
\[
\varphi\mapsto \varphi-i\int_\R e^{-it\Delta}\bigl\{[K\ast|e^{it\Delta}\varphi|^2]e^{it\Delta}\varphi\bigr\}\,dt. 
\]
In particular, evaluating $i\langle (\S-I)(\varphi),\varphi\rangle$ and invoking the Born approximation, we find that knowledge of the scattering map determines integrals of the form 
\begin{equation}\nonumber%\label{JK}
\mathcal{J}_K(\varphi):=\int_\R\int_{\R^d}\int_{\R^d} K(x-y)|e^{it\Delta}\varphi(x)|^2|e^{it\Delta}\varphi(y)|^2\,dx\,dy\,dt;
\end{equation}
see Proposition~\ref{P:Born}.  

The proof of Theorem~\ref{T} therefore reduces to showing that the functional $\mathcal{J}_K$ uniquely determines $K$. For this task, we specialize to the case of Gaussian data.  Using the fact that the linear Schr\"odinger evolution of Gaussian data remains Gaussian for all times, we are able to compute the functional $\mathcal{J}_K$ explicitly for such data. We remark that this approach has been utilized in the related works \cite{KMV, KMV2, Murphy, ChenMurphy1, ChenMurphy2, ChenMurphy3} for the case of local nonlinearities.  In the simplest setting of nonlinearities of the form $\alpha(x)|u|^p u$, one can use rescaled Gaussians and an approximate identity argument to determine the pointwise values of the coefficient $\alpha$.  In this work we show that a similar approach may be brought to bear in the nonlocal setting with radial convolution kernels.  Instead of directly accessing the pointwise values of the kernel $K$, we find that the scattering map determines the values
\[
\langle |\nabla|^{-1} K,\psi_\sigma\rangle,\qtq{where} \psi_\sigma(x):=\exp\{-\tfrac{|x|^2}{4\sigma^2}\},\quad \sigma>0. 
\]
For radial and admissible $K$, we then show (via the Laplace transform) that such Gaussian averages uniquely determine $|\nabla|^{-1} K$ and hence $K$ itself.  

While most of the physically-relevant examples of Hartree-type nonlinearities feature radial convolution kernels, it is also natural to ask whether the analogue of Theorem~\ref{T} holds in the nonradial setting.  We believe it does, and we plan to address this question in future work. 

The rest of the paper is organized as follows: In Section~\ref{S:notation}, we introduce notation and collect several useful lemmas.  In Section~\ref{S:scatter}, we prove small-data scattering for \eqref{nls} (Theorem~\ref{T:scatter}) and derive the Born approximation (Proposition~\ref{P:Born}).  In Section~\ref{S:Gaussian}, we compute the leading term in the Born approximation in the case of Gaussian data.  Finally, in Section~\ref{S:proof}, we prove the main result, Theorem~\ref{T}.

\subsection*{Acknowledgements} Much of this work was completed while L.~C. was a Visiting Scholar at the University of Oregon under the support of Conselho Nacional de Desenvolvimento Cient\'ifico e Tecnol\'ogico - CNPq. L.~C. was partially supported by CNPq grants 07733/2023-8 and 404800/2024-6, and the FAPEMIG grant APQ-03186-24. R.~K. was supported by NSF grant DMS-2452346 and the project ANR-25-CFFS-0004 ``PhysMathEDPInteg'' of the France 2030 program. J.~M. was supported by NSF grant DMS-2350225 and Simons Foundation grant MPS-TSM-00006622.  M.~V. was supported by NSF grant DMS-2348018.

%%%%%%%%%%%%%%%
%%%%%%%%%%%%%%%
%%%%%%%%%%%%%%%
%%%%%%%%%%%%%%%
\section{Notation and preliminaries}\label{S:notation}

We use the notation $A\lesssim B$ to denote $A\leq CB$ for some $C>0$.  We write $\langle \cdot,\cdot\rangle$ for the standard $L^2$ inner product, complex linear in the second argument.  We use the following convention for the Fourier transform:
\[
\F(f)(\xi):=\widehat f(\xi) :=(2\pi)^{-\frac{d}{2}}\int_{\R^d} e^{-ix\cdot\xi}f(x)\,dx 
\]
so that 
\[
f(x) =\F^{-1}(f)(x)= (2\pi)^{-\frac{d}{2}}\int_{\R^d} \!e^{ix\cdot\xi} \hat f(\xi)\,d\xi.
\]

We recall that Fourier multiplier operators may be expressed as convolution operators as follows: 
\[
[\F^{-1} m \F] f = (2\pi)^{-\frac{d}{2}} [\F^{-1} m] \ast f.
\]

The linear Schr\"odinger propagator is defined as the Fourier multiplier operator $e^{it\Delta}=\F^{-1} e^{-it|\xi|^2}\F$.

Throughout the paper we let $\psi$ denote the standard Gaussian
\begin{equation}\label{psi}
\psi(x) := \exp\{-\tfrac{|x|^2}{4}\},
\end{equation}
and we write $\psi_\sigma$ for the rescaling
\begin{equation}\label{psisigma}
\psi_\sigma(x):=\psi(\tfrac{x}{\sigma}), \quad \sigma>0.
\end{equation}
The Gaussian evolves under the linear Schr\"odinger equation as follows:
\begin{equation}\label{gaussian-solution}
e^{it\Delta}\psi(x) = (1+it)^{-\frac{d}{2}}\exp\{-\tfrac{|x|^2}{4(1+it)}\},
\end{equation}
By the scaling symmetry for the Schr\"odinger equation, we have
\begin{equation}\label{rescaled}
e^{it\Delta}\psi_\sigma(x) = [e^{i\sigma^{-2} t\Delta}\psi](\tfrac{x}{\sigma}).
\end{equation}

For $s>-d$, the operators $|\nabla|^s$ are defined as the Fourier multiplier operators $|\nabla|^s=\F^{-1} |\xi|^s\F$. The Sobolev spaces $\dot H^s(\R^d)$ and $ H^s(\R^d)$ are defined as the completion of Schwartz functions with respect to the norms
\[
\|f\|_{\dot H^s}= \||\nabla|^sf\|_{L^2} \qtq{and} \|f\|_{H^s}= \|f\|_{L^2} + \|f\|_{\dot H^s}.
\]

By direct calculation, for $\alpha\in(0,d)$ we have 
\begin{equation}\label{purepower}
2^{\frac{\alpha}{2}}\Gamma(\tfrac{\alpha}{2})|\xi|^{-\alpha} = \F\bigl[ 2^{\frac{d-\alpha}{2}}\Gamma(\tfrac{d-\alpha}{2})|x|^{-(d-\alpha)}\bigr],
\end{equation}
where $\Gamma$ denotes the Gamma function:
\begin{equation}\label{Gamma}
\Gamma(z) := \int_0^\infty r^z e^{-r}\tfrac{dr}{r},\quad \Re z>0.
\end{equation}
In particular, for $s\in(0,d)$ the operator $|\nabla|^{-s}$ is given by convolution with $C|x|^{-(d-s)}$ for some $C=C(d,s)>0$. 
%Hint: Compute $\int_0^{\infty} e^{-t|x|^2} t^{\frac{d-\alpha}{2}} \tfrac{dt}{t} as well as its Fourier transform.
%and
%\[
%\F^{-1}[FG] = (2\pi)^{-\frac{d}{2}}\check F \ast \check G. 
%\]

We will employ the standard Strichartz estimates for the linear Schr\"odinger propagator.  We call a pair $(q,r)\in[2,\infty]\times[2,\infty]$ \emph{Schr\"odinger admissible} in $d$ dimensions if $\tfrac{2}{q}+\tfrac{d}{r}=\tfrac{d}{2}$ and $(d,q,r)\neq(2,2,\infty)$.

\begin{lemma}[Strichartz estimates, \cite{GinibreVelo, KeelTao, Strichartz}]\label{L:S} For any Schr\"odinger admissible pair $(q,r)$ and any $\varphi\in L^2(\R^d)$, 
\[
\|e^{it\Delta}\varphi\|_{L_t^q L_x^r(\R\times\R^d)}\lesssim \|\varphi\|_{L^2}. 
\]
Given an interval $I\ni t_0$, Schr\"odinger admissible pairs $(q,r),(\tilde q,\tilde r)$, and $F\in L_t^{\tilde q'}L_x^{\tilde r'}(I\times\R^d)$, we have
\[
\biggl\| \int_{t_0}^t e^{i(t-s)\Delta}F(s)\,ds \biggr\|_{L_t^q L_x^r(I\times\R^d)}\lesssim \|F\|_{L_t^{\tilde q'}L_x^{\tilde r'}(I\times\R^d)}. 
\]
\end{lemma}

The following theorem is a direct consequence of \cite[Theorem~1]{GrafakosOh}, which provides a generalization of the usual fractional product rule to a wider range of exponents and powers of $|\nabla|$. 

\begin{theorem}[Fractional product rule]\label{T:FL} Let $1\leq r<\infty$ and $1<p_1,p_2,q_1,q_2\leq\infty$ satisfy
\[
\tfrac{1}{r}=\tfrac{1}{p_1}+\tfrac{1}{q_1}=\tfrac{1}{p_2}+\tfrac{1}{q_2}.
\]
For any $s>0$ and any Schwartz functions $f$ and $g$, 
\[
\| |\nabla|^s(fg)\|_{L^r(\R^d)} \lesssim \| |\nabla|^s f\|_{L^{p_1}(\R^d)}\|g\|_{L^{q_1}(\R^d)} + \|f\|_{L^{p_2}(\R^d)}\||\nabla|^s g\|_{L^{q_2}(\R^d)}.
\]
\end{theorem}

We will also need the following technical lemma concerning the inverse gradient of an admissible kernel.

\begin{lemma}\label{L:inverse-K} Suppose $K$ is admissible in the sense of Definition~\ref{D:admissible} with parameter $c\in[2,d)$.  Then $|\nabla|^{-1}K$ is continuous on $\R^d\backslash\{0\}$ and satisfies
\begin{equation}\label{inverse-K-bounds}
||\nabla|^{-1} K(z)| \lesssim |z|^{-1}+|z|^{-c+1}. 
\end{equation}
\end{lemma}

\begin{proof}
By \eqref{purepower}, for any $z\neq 0$,
\[
|\nabla|^{-1} K(z) = C[L\ast K](z), \qtq{where} L(y):=|y|^{1-d}. 
\] 

We now let $B$ denote the unit ball in $\R^d$ and decompose
\[
K=K_2+K_c,\qtq{where} K_2=K[1-\chi_B] \qtq{and} K_c=K\chi_B. 
\]
By Definition~\ref{D:admissible}, we have $|K_a(x)|\lesssim |x|^{-a}$ for $a\in\{2,c\}$.  The proof now reduces to proving that for $a\in\{2,c\}$, $L\ast K_a$ is continuous on $\R^d\backslash\{0\}$ and satisfies
\[
|L\ast K_a(z)| \lesssim |z|^{-a+1} \qtq{for} z\in\R^d\backslash\{0\}.
\]

We fix $a\in\{2,c\}$ and $\delta>0$. We let $\chi_{\leq\delta}$ denote the characteristic function of $\{|z|\leq\delta\}$ and write $\chi_{>\delta}=1-\chi_{\leq\delta}$. We set $K_{\leq\delta}=K_a\chi_{\leq\delta}$ and define $K_{>\delta}$, $L_{\leq\delta}$, and $L_{>\delta}$ analogously. 

Support considerations yield that for $|z|\geq 4\delta$, we have
\begin{equation}\label{LastK}
[L\ast K_a](z)  = [L_{>\delta}\ast K_{>\delta}](z) + [L_{>\delta}\ast K_{\leq\delta}](z) + [L_{\leq \delta}\ast K_{>\delta}](z). 
\end{equation}
As the convolution of two functions in dual $L^p$-spaces is continuous, \eqref{LastK} expresses $L\ast K_a$ as the sum of three continuous functions in the region $|z|\geq 4\delta$.  Indeed, choosing an exponent $r\in(\tfrac{d}{a},d)$, we may estimate 
\begin{align}\label{11:47}
|L\ast K_a(z)| & \leq \|L_{>\delta}\|_{L^{r'}}\|K_{>\delta}\|_{L^r} + \|L_{>\delta}\|_{L^\infty} \|K_{\leq\delta}\|_{L^1} + \|L_{\leq\delta}\|_{L^1}\|K_{>\delta}\|_{L^\infty} \notag\\
& \lesssim \delta^{-a+1}
\end{align}
uniformly for $|z|\geq 4\delta$. As $\delta>0$ was arbitrary, this shows that $L\ast K_a$ is continuous on $\R^d\backslash\{0\}$.  Finally, taking $\delta=\tfrac14|z|$ in \eqref{11:47} yields 
\[
\bigl||\nabla|^{-1}K_a(z)\bigr| \lesssim |z|^{-a+1}
\]
and \eqref{inverse-K-bounds} follows.  \end{proof}

%%%%%%%%%%%%%%%
%%%%%%%%%%%%%%%
%%%%%%%%%%%%%%%
%%%%%%%%%%%%%%%
\section{Small-data scattering and the Born approximation}\label{S:scatter}

In this section we first prove a small-data scattering result for admissible convolution kernels and define the scattering map. We will then analyze the accuracy of the Born approximation to the scattering map. 

\begin{theorem}\label{T:scatter} Fix $d\geq 3$ and let $K$ be admissible in the sense of Definition~\ref{D:admissible} with parameter $c\in[2,d)$.  Define
\[
\mathcal{B}_\eta=\{f\in H^{s_c}(\R^d): \|f\|_{H^{s_c}(\R^d)}<\eta\},\qtq{where} s_c:=\tfrac{c}{2}-1.
\]
There exists $\eta>0$ sufficiently small so that for any $u_-\in \mathcal{B}_\eta$, there is a unique solution $u:\R\times\R^d\to\C$ to \eqref{nls} and a unique $u_+\in H^{s_c}(\R^d)$ such that
\[
\| |\nabla|^\ell u \|_{L_t^4 L_x^{\frac{2d}{d-1}}(\R\times\R^d)} \lesssim \|u_-\|_{\dot H^\ell(\R^d)}\qtq{for}\ell\in\{0,s_c\}
\]
and
\[
\lim_{t\to\pm\infty}\|u(t)-e^{it\Delta} u_\pm \|_{H^{s_c}(\R^d)} = 0. 
\]
\end{theorem}

\begin{definition}\label{D:scatter} Given an admissible kernel $K$ with parameter $c\in[2,d)$, we may use Theorem~\ref{T:scatter} to define the \emph{scattering map} $\S:\mathcal{B}_\eta\to H^{s_c}(\R^d)$ by $\S(u_-)=u_+$. 
\end{definition}

\begin{proof} For $\eta>0$ to be specified below and $u_-\in \mathcal{B}_\eta$, we write
\[
\|u\|_X = \|u\|_{L_t^4 L_x^{\frac{2d}{d-1}}(\R\times\R^d)} \qtq{and} \|F\|_{X'}=\|F\|_{L_t^{\frac43} L_x^{\frac{2d}{d+1}}(\R\times\R^d)}
\]
and define the complete metric space $(Y,d)$ by
\[
Y=\{u:\R\times\R^d\to\C: \|u\|_X \leq 2C\|u_-\|_{L^2},\quad \||\nabla|^{s_c} u\|_X \leq 2C\||\nabla|^{s_c}u_-\|_{L^2}\},
\]
equipped with the metric
\[
d(u,v) = \|u-v\|_X.
\]
Here $C\geq1$ is a universal constant that encodes the implicit constants appearing in the estimates below. 

We now define
\[
\Phi(u) := e^{it\Delta}u_- - i\int_0^t e^{i(t-s)\Delta} F(u(s))\,ds, \quad F(u):= (K\ast |u|^2)u.
\]
We will show that $\Phi:Y\to Y$ and that $\Phi$ is a contraction for $\eta$ small enough.

Let $B$ denote the unit ball in $\R^d$.  In what follows, it will be convenient to decompose 
\[
K=K_2 + K_c, \qtq{where}K_2=K[1-\chi_B]\qtq{and}K_c=K\chi_B.
\]
By Definition~\ref{D:admissible}, we then have that $|K_a(x)| \lesssim |x|^{-a}$, so that $K_a\in L^{\frac{d}{a},\infty}$ for $a\in\{2,c\}$.

Let $u\in Y$. Throughout the proof, all space-time norms will be taken over $\R\times\R^d$.  Using Strichartz estimates, the Hardy--Littlewood--Sobolev inequality, and Sobolev embedding, we may bound
\begin{align*}
\|\Phi(u)\|_X & \lesssim \|u_-\|_{L^2} +\sum_{a\in\{2,c\}} \|(K_a\ast |u|^2)u\|_{X'} \\
& \lesssim \|u_-\|_{L^2} + \sum_{a\in\{2,c\}} \| |x|^{-a}\ast |u|^2\|_{L_t^2 L_x^d}\|u\|_X \\
& \lesssim \|u_-\|_{L^2} + \sum_{a\in\{2,c\}} \||x|^{-a}\|_{L^{\frac{d}{a},\infty}} \|u\|_{L_t^4 L_x^{\frac{2d}{d+1-a}}}^2\|u_-\|_{L^2} \\
& \lesssim \|u_-\|_{L^2} + \sum_{a\in\{2,c\}} \| |\nabla|^{\frac{a}{2}-1} u\|_X^2 \|u_-\|_{L^2} \\
& \lesssim \|u_-\|_{L^2} + \eta^2 \|u_-\|_{L^2} \\
& \leq 2C\|u_-\|_{L^2},
\end{align*}
provided $\eta$ is sufficiently small. 

We turn to the estimate with derivatives. We first apply Lemma~\ref{L:S} to obtain
\begin{align*}
\| |\nabla|^{s_c} \Phi(u)\|_X \lesssim \| |\nabla|^{s_c} u_-\|_{L^2} + \sum_{a\in\{2,c\}} \| |\nabla|^{s_c}[(K_a\ast |u|^2)u]\|_{X'}.
\end{align*}
To bound the nonlinear terms we will use Theorem~\ref{T:FL} and H\"older's inequality.  For $a\in\{2,c\}$, we estimate
\begin{align*}
\| |&\nabla|^{s_c}[(K_a\ast |u|^2)u]\|_{X'} \\
& \lesssim \| K_a\ast[|\nabla|^{s_c} |u|^2]\|_{L_t^2 L_x^{\frac{2d}{a}}} \|u\|_{L_t^4 L_x^{\frac{2d}{d+1-a}}} + \| K_a\ast |u|^2\|_{L_t^2 L_x^d} \| |\nabla|^{s_c} u\|_{X} \\
& \lesssim \| |x|^{-a}\|_{L^{\frac{d}{a},\infty}}\bigl\{ \| |\nabla|^{s_c} |u|^2\|_{L_t^2 L_x^{\frac{2d}{2d-a}}} \| u\|_{L_t^4 L_x^{\frac{2d}{d+1-a}}} + \|u\|_{L_t^4 L_x^{\frac{2d}{d+1-a}}}^2 \| |\nabla|^{s_c} u\|_X\bigr\} \\
& \lesssim \|u\|_{L_t^4 L_x^{\frac{2d}{d+1-a}}}^2 \| |\nabla|^{s_c} u\|_X \\
& \lesssim \|  |\nabla|^{\frac{a}{2}-1} u\|_X \| |\nabla|^{s_c} u\|_X \\
& \lesssim \eta^2 \||\nabla|^{s_c} u_-\|_{L^2}. 
\end{align*}
Continuing from above, we derive
\[
\||\nabla|^{s_c}\Phi(u)\|_X \lesssim \||\nabla|^{s_c} u_-\|_{L^2} + \eta^2 \||\nabla|^{s_c} u_-\|_{L^2} \leq 2C\| |\nabla|^{s_c} u_-\|_{L^2},
\]
provided $\eta$ is sufficiently small.  This shows that $\Phi:Y\to Y$. 

To see that $\Phi$ is a contraction, we let $u,v\in Y$ and first apply Strichartz estimates to obtain
\[
\|\Phi(u) - \Phi(v)\|_X \lesssim \|F(u)-F(v)\|_{X'}.
\]
We now write
\begin{equation}\label{F-diff}
F(u) - F(v) = (K\ast |u|^2)(u-v)+[K\ast\{(|u|+|v|)(|u|-|v|)\}]v.
\end{equation}
For the first term, we argue as above to obtain
\begin{align*}
\| (K\ast |u|^2)(u-v)\|_{X'} & \lesssim \sum_{a\in\{2,c\}} \| |\nabla|^{\frac{a}{2}-1} u\|_{X}^2 \|u-v\|_X \\
& \lesssim \eta^2 \|u-v\|_X.
\end{align*} 
For the second term, we estimate similarly:
\begin{align*}
\| [K\ast&\{(|u|+|v|)(|u|-|v|)\}]v\|_{X'} \\
& \lesssim \sum_{a\in\{2,c\}} \| |x|^{-a}\ast\{(|u|+|v|)(|u|-|v|)\}\|_{L_t^2 L_x^{\frac{2d}{a}}} \|v\|_{L_t^4 L_x^{\frac{2d}{d+1-a}}} \\
& \lesssim \sum_{a\in\{2,c\}} \| |x|^{-a}\|_{L^{\frac{d}{a},\infty}} \||u|+|v|\|_{L_t^4 L_x^{\frac{2d}{d+1-a}}} \|u-v\|_{L_t^4 L_x^{\frac{2d}{d-1}}} \| v\|_{L_t^4 L_x^{\frac{2d}{d+1-a}}} \\
& \lesssim\sum_{a\in\{2,c\}} \bigl(\||\nabla|^{\frac{a}{2}-1}u\|_X +\||\nabla|^{\frac{a}{2}-1}v\|_X\bigr)\| |\nabla|^{\frac{a}{2}-1} v\|_X \|u-v\|_X \\
& \lesssim \eta^2 \|u-v\|_X.
\end{align*}
Thus
\[
\|\Phi(u)-\Phi(v)\|_X \lesssim \eta^2 \|u-v\|_X \leq \tfrac12 \|u-v\|_X,
\]
provided $\eta$ is sufficiently small.

It follows that $\Phi$ is a contraction on $Y$ and so that there exists a unique $u\in Y$ satisfying $u=\Phi(u)$.  This yields the desired solution $u$ obeying the stated space-time bounds.

It remains to prove scattering forward in time.  To this end, it suffices to prove that $\{e^{-it\Delta}u(t)\}$ is Cauchy in $H^{s_c}$ as $t\to\infty$. In fact, repeating the estimates above and writing $X(s,t)$ to denote $L_t^4 L_x^{\frac{2d}{d-1}}((s,t)\times\R^d)$, we may obtain
\begin{align*}
\| e^{-it\Delta}u(t)-e^{-is\Delta}u(s)\|_{H^{s_c}} & \lesssim \|F(u)\|_{X'(s,t)} + \| |\nabla|^{s_c} F(u)\|_{X'(s,t)} \\
& \lesssim \Bigl( \|u\|_{X(s,t)} + \| |\nabla|^{s_c} u\|_{X(s,t)}\Bigr)^3  \to 0 \qtq{as}s,t\to\infty.
\end{align*}
This completes the proof of the theorem. \end{proof}

%%%%%%%%%%%%%%%
%%%%%%%%%%%%%%%
%%%%%%%%%%%%%%%
%%%%%%%%%%%%%%%
 We next derive the Born approximation for the scattering map, which entails replacing the full solution by its first Picard iterate.  This allows us to approximate the implicitly-defined scattering map by an explicit functional. 

\begin{proposition}[Born approximation]\label{P:Born} Let $K$ be admissible in the sense of Definition~\ref{D:admissible}, with parameter $c\in[2,d)$. Define $s_c=\tfrac{c}{2}-1$ and let $\S:\mathcal{B}_\eta\subset H^{s_c}(\R^d)\to H^{s_c}(\R^d)$ denote the scattering map defined in Definition~\ref{D:scatter}. 

For any $\varphi\in \mathcal{B}_\eta$, 
\begin{align*}
i\langle (\S-I)(\varphi),\varphi\rangle & = \iiint K(x-y)|e^{it\Delta}\varphi(x)|^2 |e^{it\Delta}\varphi(y)|^2\,dx\,dy\,dt \\
& \quad + \mathcal{O}\bigl\{\|\varphi\|_{L^2}^2\bigl(\|\varphi\|_{L^2}^4 + \|\varphi\|_{\dot H^{s_c}}^4\bigr)\bigr\}.
\end{align*}
\end{proposition}

\begin{proof} Let $K$, $\S$, and $\varphi\in \mathcal{B}_\eta$ be as in the statement of the proposition. Let $u$ denote the solution to \eqref{nls} constructed in Theorem~\ref{T:scatter} that satisfies $e^{-it\Delta}u(t)\to \varphi$ as $t\to-\infty$.  Using Theorem~\ref{T:scatter}, we may write
\[
S(\varphi) = \varphi - i\iint e^{it\Delta}[(K\ast |u(t)|^2)u(t)]\,dx\,dt. 
\]
Thus it suffices to prove
\begin{equation}\label{Born-ets}
\biggl| \iint [(K\ast |u|^2)u-(K\ast |e^{it\Delta}\varphi|^2)e^{it\Delta}\varphi] \overline{e^{it\Delta}\varphi}\,dx\,dt\biggr| \lesssim \|\varphi\|_{L^2}^2\bigl(\|\varphi\|_{L^2}^4 + \|\varphi\|_{\dot H^{s_c}}^4\bigr).
\end{equation}
As in the proof of Theorem~\ref{T:scatter}, we denote $F(u)=(K\ast |u|^2)u$ and use \eqref{F-diff}. Estimating as in the proof of Theorem~\ref{T:scatter} (and employing the notation from that proof as well), we have
\begin{align*}
\text{LHS}\eqref{Born-ets} & \lesssim \|e^{it\Delta}\varphi\|_{X} \|F(u) -F(e^{it\Delta}\varphi)\|_{X'} \\
& \lesssim \|\varphi\|_{L^2}\bigl( \|\varphi\|_{L^2}^2+\|\varphi\|_{\dot H^{s_c}}^2\bigr) \|u-e^{it\Delta}\varphi\|_{X} \\
& \lesssim \|\varphi\|_{L^2}\bigl(\|\varphi\|_{L^2}^2+\|\varphi\|_{\dot H^{s_c}}^2\bigr) \|F(u)\|_{X'} \\
& \lesssim  \|\varphi\|_{L^2}^2\bigl(\|\varphi\|_{L^2}^2+\|\varphi\|_{\dot H^{s_c}}^2\bigr)^2,
\end{align*}
which yields the result. \end{proof}

%%%%%%%%%%%%%%%
%%%%%%%%%%%%%%%
%%%%%%%%%%%%%%%
%%%%%%%%%%%%%%%

\section{Born approximation on Gaussian inputs}\label{S:Gaussian}

To prove the main result (Theorem~\ref{T}), we will apply the Born approximation (Proposition~\ref{P:Born}) with rescaled Gaussian data $\varphi(x) = \eps \psi_\sigma(x)$ for some $\eps,\sigma>0$.  In this section, we compute explicitly the leading order term in the Born approximation with Gaussian data. 

We first recall the standard Gaussian integral
\begin{equation}\label{Gaussian-integral}
\int_{\R^d} e^{-\alpha|x|^2}\,dx = (\tfrac{\pi}{\alpha})^{\frac{d}{2}},\quad \Re \alpha>0.
\end{equation} 
We will also use the following recapitulation of \eqref{Gamma}:
\begin{equation}\label{Gamma-id}
|y|^{-2z} = \tfrac{1}{\Gamma(z)}\int_0^\infty r^z e^{-r|y|^2}\tfrac{dr}{r},\quad \Re z>0.
\end{equation}

The main result of this section is the following identity. 

\begin{proposition}[Born approximation on Gaussian inputs]\label{P1} Let $K:\R^d\backslash\{0\}\to\R$ be admissible with parameter $c\in[2,d)$.  For any $\sigma>0$, 
\begin{equation}\label{EP1}
\iiint K(z-w) |e^{it\Delta}\psi_\sigma(z)|^2 |e^{it\Delta} \psi_\sigma(w)|^2\,dz\,dw\,dt = (\pi\sigma^2)^{\frac{d+1}{2}}\langle |\nabla|^{-1} K,\psi_\sigma\rangle,
\end{equation}
where $\psi_\sigma$ is defined by \eqref{psi} and \eqref{psisigma}.
\end{proposition}

The proof of Proposition~\ref{P1} will rely on the following explicit integral:
 
\begin{lemma}\label{L1} For any $x\in\R^d$, 
\begin{equation}\label{one-Gaussian-ID}
\int_\R (1+t^2)^{-\frac{d}{2}} \exp\bigl\{ -\tfrac{|x|^2}{4(1+t^2)}\bigr\} \,dt  = \pi^{\frac12}[ |\nabla|^{-1} \psi](x).
\end{equation}
\end{lemma}

\begin{proof} We first change variables $t=r^{-\frac12}$ to obtain
\begin{align*}
\text{LHS}\eqref{one-Gaussian-ID} & = 2\int_0^\infty (1+t^2)^{-\frac{d}{2}}\exp\{-\tfrac{|x|^2}{4(1+t^2)}\}\,dt \\
& = \int_0^\infty (1+r)^{-\frac{d}{2}}r^{\frac{d-1}{2}}\exp\{-\tfrac{r|x|^2}{4(1+r)}\}\,\tfrac{dr}{r}. 
\end{align*}

We now use \eqref{Gaussian-integral} to write
\[
(1+r)^{-\frac{d}{2}} = \pi^{-\frac{d}{2}} \int_{\R^d} \exp\{-(1+r)|y-\tfrac{x}{2(1+r)}|^2\}\,dy. 
\]

Inserting this expression into the identity above, simplifying the exponent, and applying \eqref{Gamma-id} with $z=\tfrac{d-1}{2}$, we obtain
\begin{align*}
\text{LHS}\eqref{one-Gaussian-ID} & = \pi^{-\frac{d}{2}}  \int_{\R^d}\biggl[\int_0^\infty r^{\frac{d-1}{2}}\exp\{-r|y|^2\}\tfrac{dr}{r}\biggr]\exp\{-|y-\tfrac{x}{2}|^2\} \,dy\\
& = \pi^{-\frac{d}{2}}\Gamma(\tfrac{d-1}{2}) \int_{\R^d} |y|^{-(d-1)}\exp\{-|y-\tfrac{x}{2}|^2\}\,dy  \\
& = \tfrac12\pi^{-\frac{d}{2}}\Gamma(\tfrac{d-1}{2})\int_{\R^d} |y|^{-(d-1)}\exp\{-\tfrac{|x-y|^2}{4}\}\,dy.
\end{align*}
We now use \eqref{purepower} with $\alpha=1$ to see that
\[
\tfrac12\pi^{-\frac{d}{2}}\Gamma(\tfrac{d-1}{2})|y|^{-(d-1)} = \pi^{\frac12}\cdot(2\pi)^{-\frac{d}{2}} \F^{-1}(|\xi|^{-1}),
\]
which yields \eqref{one-Gaussian-ID}. \end{proof}

\begin{proof}[Proof of Proposition~\ref{P1}] We begin by using \eqref{gaussian-solution}, \eqref{rescaled}, and a change of variables to obtain
\begin{align*}
\text{LHS}\eqref{EP1} = \sigma^{2d+2}\iiint \tfrac{K(\sigma(z-w))}{(1+t^2)^d} \exp\{-\tfrac{|z|^2+|w|^2}{2(1+t^2)}\}\,dz\,dw\,dt.
\end{align*}
Next, we use the change of variables
\[
(x,y)=(z-w,z+w)
\]
followed by \eqref{Gaussian-integral} and Lemma~\ref{L1}.  This yields
\begin{align*}
\text{LHS}\eqref{EP1} & =2^{-d}\sigma^{2d+2} \iiint \tfrac{K(\sigma x)}{(1+t^2)^d} \exp\{-\tfrac{|x|^2+|y|^2}{4(1+t^2)}\}\,dx\,dy \,dt \\
& = \pi^{\frac{d}{2}}\sigma^{2d+2} \int K(\sigma x)\int (1+t^2)^{-\frac{d}{2}}\exp\{-\tfrac{|x|^2}{4(1+t^2)}\}\,dt\,dx \\
&  = \pi^{\frac{d+1}{2}}\sigma^{2d+2}\int K(\sigma x)[|\nabla|^{-1} \psi](x)\,dx \\
& = \pi^{\frac{d+1}{2}}\sigma^{d+1}\langle |\nabla|^{-1} K,\psi_\sigma\rangle,
\end{align*}
which yields \eqref{EP1}. \end{proof}

%%%%%%%%%%%%%%%
%%%%%%%%%%%%%%%
%%%%%%%%%%%%%%%
%%%%%%%%%%%%%%%
\section{Proof of the main result}\label{S:proof}

In this section we complete the proof of the main result, Theorem~\ref{T}. 

\begin{proof}[Proof of Theorem~\ref{T}] Let $K_1$ and $K_2$ be admissible kernels in the sense of Definition~\ref{D:admissible} with parameters $c_1,c_2\in[2,d)$, respectively. As admissibility with one parameter implies admissibility with any larger parameter, we may assume without loss of generality that $c_1=c_2=:c$. Writing $s_c=\tfrac{c}{2}-1$ and recalling Theorem~\ref{T:scatter}, let us denote the corresponding scattering maps by $\S_1,\S_2:\mathcal{B}\subset H^{s_c}(\R^d)\to H^{s_c}(\R^d)$, where $\mathcal{B}$ is a sufficiently small ball. We assume $\S_1|_{\mathcal{B}}\equiv \S_2|_{\mathcal{B}}$ and seek to prove that $K_1= K_2$ almost everywhere. 

To keep the formulas within the margins, we write
\[
W:=K_2-K_1.
\]
We fix $0<\sigma\leq 1$ and let $\eps>0$ to be determined below.  Recalling \eqref{psi} and \eqref{psisigma}, we have that
\[
\|\eps\psi_\sigma\|_{H^{s_c}} \lesssim \eps \sigma^{1+\frac{d-c}{2}},
\]
so that $\eps\psi_\sigma\in\mathcal{B}$ for all $\eps=\eps(\sigma)$ sufficiently small. 

Applying the Born approximation (Proposition~\ref{P:Born}), we obtain
\begin{equation}\label{proof1}
\begin{aligned}
0 & = \langle(\S_2-\S_1)(\eps\psi_\sigma),\eps\psi_\sigma \rangle \\
& = \eps^4\iiint W(z-w)|e^{it\Delta}\psi_\sigma(z)|^2 |e^{it\Delta}\psi_\sigma(w)|^2\,dz\,dw\,dt + \mathcal{O}(\eps^6 \sigma^{3d+4-2c}). 
\end{aligned}
\end{equation} 
On the other hand, by Proposition~\ref{P1},
\begin{equation}\label{proof2}
\iiint W(z-w)|e^{it\Delta}\psi_\sigma(z)|^2 |e^{it\Delta}\psi_\sigma(w)|^2\,dz\,dw\,dt =  (\pi\sigma^2)^{\frac{d+1}{2}}\langle |\nabla|^{-1} W,\psi_\sigma\rangle. 
\end{equation}
Combining \eqref{proof1} and \eqref{proof2}, we conclude that
\[
\langle |\nabla|^{-1}W,\psi_\sigma\rangle = \mathcal{O}(\eps^2\sigma^{2d+3-2c}).
\]
Sending $\eps\to 0$ for each fixed $\sigma>0$ then yields
\begin{equation}\label{proof3}
\langle |\nabla|^{-1}W,\psi_\sigma\rangle = 0 \qtq{for all}\sigma>0.
\end{equation}
Writing $|\nabla|^{-1} W$ as a function of the radial variable $r>0$, using spherical coordinates, and changing variables via $\rho=r^2$, it follows from \eqref{proof3} that
\begin{equation}\label{proof4}
\int_0^\infty \rho^{\frac{d-2}{2}}[|\nabla|^{-1} W](\sqrt{\rho}) e^{-s\rho}\,d\rho = 0 \qtq{for all}s>0. 
\end{equation}
We now observe that the left-hand side of \eqref{proof4} is the Laplace transform of
\[
V(\rho):=\rho^{\frac{d-2}{2}}[|\nabla|^{-1} W](\sqrt{\rho}).
\]
Lemma~\ref{L:inverse-K} guarantees that $V$ is continuous on $(0,\infty)$ and satisfies 
\[
|V(\rho)| \lesssim \rho^{\frac{d-3}{2}}+\rho^{\frac{d-1-c}{2}}.
\]
Thus, by the injectivity of the Laplace transform, \eqref{proof4} implies $|\nabla|^{-1} W\equiv 0$, i.e.
\[
|\nabla|^{-1} K_1 \equiv |\nabla|^{-1} K_2 .
\]
The bounds of Lemma~\ref{L:inverse-K} allow us to extend both sides (identically!) as harmonic functions in the upper half-space via the Poisson integral.  One may then recover $K_1$ and $K_2$ as the normal derivatives at the boundary and so deduce that $K_1=K_2$ almost everywhere.\end{proof}

%%%%%%%%%%%%%%%
%%%%%%%%%%%%%%%
%%%%%%%%%%%%%%%
%%%%%%%%%%%%%%%

\end{document}